\documentclass[12pt,french]{article}
\usepackage{graphicx}
\usepackage{a4} 
\usepackage[T1]{fontenc} 
\usepackage[cyr]{aeguill} 
\usepackage{epsfig} 
\usepackage{amsmath, amsthm} 
\usepackage{amsfonts,amssymb}
\usepackage{float} 
\usepackage{url} 
\usepackage{colortbl}
\usepackage{graphics}
\usepackage[latin1]{inputenc}
\usepackage{tikz,tkz-tab}
\usepackage{fancyhdr}
\newtheorem{remark}{Remark}[section]
\newtheorem{definition}{Definition}[section]
\newtheorem{proposition}{Proposition}[section]
\newtheorem{lemma}{Lemma}[section]

\newtheorem{theorem}{Theorem}[section]
\newtheorem{corollary}{Corollary}[section]
\date{}

\begin{document}
\title{Existence and symmetry results for some overdetermined free boundary problems}
\author{Mohammed Barkatou}
\date{}
\maketitle
\begin{abstract}
In this paper, we prove that a domain which verifies some integral inequality is
either (strictly) contained in the solution of some free boundary problem, or it coincides with an $N$-ball.
We also present new overdetermined value problems which have an $N$-ball as a solution.
To reach our results, we use an integral identity which involves the domain derivative of the solution of the corresponding Dirichlet problem.
\end{abstract}

\maketitle {\small {\noindent {\bf Keywords:} free boundaries, maximum principle, domain derivative, Laplace operator
\noindent {\bf 2010 MSC} 35A15, 35J65, 35B50

\section{Introduction}
Assuming throughout that: $D\subset \mathbb{R}^{N}\quad (N\geq 2)$
is a bounded ball which contains all the domains we use. If
$\omega$ is an open subset of $D$, let $\nu$ be the outward normal
to $\partial\omega$ and let $\left| \partial \omega \right|$
(respectively $\left| \omega \right|$) be the perimeter
(respectively the volume) of $\omega$. Let $f$ be a positive function
belonging to $L^2(\mathbb{R}^{N})$ and having a compact support
$K$ with nonempty interior. Denote by $C$ the convex hull of $K$.\\

Consider the following overdetermined boundary value problems.
\begin{equation*}
\mathcal{QS}(f,k) \left\{
\begin{array}{c}
-\Delta u_{\Omega}=f \quad \text{in  }\Omega\,, \\
u_{\Omega}=0\text{ on  }\partial \Omega\,, \\
 -\frac{\partial u_{\Omega}}{\partial \nu }=k=\text{cst.}\text{  on  }\partial
 \Omega.
\end{array}
\right.
\end{equation*}
and

\begin{equation*}
\mathcal{P}(f,c) \left\{
\begin{array}{c}
-\Delta u_{\Omega}=f \quad \text{ in  }\Omega\,,
 u_{\Omega}=0\text{  on  }\partial\Omega \quad\big\}P(\Omega,f)\,, \\
-\Delta v_{\Omega}=u_{\Omega} \quad \text{ in  }\Omega\,,
v_{\Omega}=0 \text{ on  }\partial \Omega \quad\big\}P(\Omega,u_{\Omega})\,,
\\
|\nabla u_{\Omega}||\nabla v_{\Omega}|=c=\text{cst.}\text{  on
}\partial\Omega.

\end{array}
\right.
\end{equation*}

\begin{remark}
The second problem is equivalent to the following biharmonic one:
\begin{equation*}
\left\{
\begin{array}{c}
\Delta^{2} v_{\Omega}=f \quad \text{ in  }\Omega\,,\\
 v_{\Omega}=\Delta v_{\Omega}=0\text{  on  }\partial\Omega\,, \\
|\nabla u_{\Omega}||\nabla v_{\Omega}|=c\text{  on
}\partial\Omega.

\end{array}
\right.
\end{equation*}
\end{remark}

Notice that since $u_{\Omega}$ (resp. $v_{\Omega}$) vanishes on $\partial\Omega$ then
$-\frac{\partial u_{\Omega}}{\partial \nu }=|\nabla u_{\Omega}|$
($-\frac{\partial v_{\Omega}}{\partial \nu }=|\nabla v_{\Omega}|$).\\

The problem $\mathcal{QS}(f,k)$ is called the quadrature surfaces free
boundary problem and arises in many areas of physics (free
streamlines, jets, Hele-show flows, electromagnetic shaping,
gravitational problems etc.) It has been intensively studied from
different points of view, by several authors. For more details
about the methods used for solving this problem see the
\cite[Introduction]{gs}. Imposing boundary conditions for both $u_{\Omega }$ and $|\nabla
u_{\Omega }|$ on $\partial\Omega$ makes problem $\mathcal{QS}(f,k)$
\textit{overdetermined}, so that in general without any
assumptions on data this problem has no solution. Gustafsson and Shahgholian \cite{gs}
conclude their paper by giving \cite[Theorem 4.7]{gs} the
following sufficient condition: If $\text{Supp}f\subset B_R$ and
if $\int_{B_R}f(x)dx>(\frac{6^NNc}{3R}|B_R|)$ with
$B_{3R}\subset\Omega_u$ ($B_R$ being some ball of radius $R$) then
$\mathcal{QS}(f,k)$ has a solution. The method used by Gustafsson and
Shahgholian goes back to K. Friedrichs \cite{fr}, or even to T. Carleman \cite{ca}, and was
considerably developed by H. W. Alt and L. A. Caffarelli
\cite{ac}. Recently, by combining the maximum principle to the
compatibility condition of the Neumann problem, Barkatou et al.
\cite{ba2} gave, $|\nabla u_C|>k\;\text{on}\;\partial C$ as a sufficient condition of existence for
$Q_S(f,g)$. Later, Barkatou \cite{ba1} showed that this problem admits a solution if and only if
the condition $\int_C f(x)dx>k|\partial C|$ is valid.\\

In 1971, Serrin \cite{se} proved that if Problem $\mathcal{QS}(1,k)$ has
a solution $u_{\Omega}\in C^{2}( \overline{\Omega }$) then
$\Omega$ must be an $N$-ball and $u_{\Omega}$ is radially symmetric. The
method used by Serrin combines the maximum principle together with
the device of moving planes \cite{gwn} to a critical position and
then showing that the solution is symmetric about the limiting
plane. In the same year, Weinberger \cite{w} gave a simplified
proof for this problem. His strategy of proof consists first on
showing that $\left| \nabla u\right| ^{2}+\frac{2}{N}u=k^{2}$ in
$\Omega$ and to derive a radial symmetry from this. A method which does not need the
maximum principle was developped by Payne and Schaefer \cite{Pa1}.
They developed integral identities which are equivalent to the problems they considered and led
to the conclusion that the domain $\Omega$ must be an $N$-ball. An other technique which does not
involve the maximum principle was introduced by Brock and Henrot \cite{bh}
(see also \cite{hc} or \cite{bk}). It consists on using the domain derivative
to get the same conclusion. For more details about the
symmetry results see \cite[Introduction]{FGK} and the references
therein. Fragal\`{a} et al \cite{FGK}, obtained their symmetry
result by combining the maximum principle for a suitable
$P$-function with some geometric arguments involving the mean
curvature of $\partial\Omega$.\\
The problem $\mathcal{P}(1,c)$ arises from variational problem in
Probability \cite{fm,km}. Fromm and McDonald \cite{fm} related
this problem to the fundamental result of Serrin. Then using the
moving plane method combining with Serrin's boundary point Lemma,
they showed that if this problem admits a solution $\Omega$ then
it must be an $N$-ball. Huang and Miller \cite{hm} established the
variational formulas for maximizing the functionals (they
considered) over $C^{k}$ domains with a volume constraint and
obtained the same symmetry result for their maximizers.\\

The problem $\mathcal{P}(f,c)$ was first studied in \cite{kb}. In general,
without any assumptions on data, the problem $\mathcal{P}(f,c)$ has no solution.
In \cite{kb}, by using the maximum principle, the authors showed that
if $|\nabla u_{C}||\nabla v_{C}|>c\text{  on }\partial C$, then
this problem has a solution. The aim here is to give a sufficient condition of
existence for the problem $\mathcal{P}(f,c)$ better than the earlier (see Section 3.).\\

In the present paper, we will use some integral inequality on $\Omega$ verified by $f$ and $k$ (respectively $c$) to
prove that the domain $\Omega$ is either (strictly) contained in a solution of the problem $\mathcal{QS}(f,\mu)$
(respectively $\mathcal{P}(f,\mu)$) for some constant $\mu$ or $\Omega$ is an $N$-ball.\\
Next, we will use integral identities involving the domain derivative of the solution of the Dirichlet problem
in order to show that the solution of a new overdetermined value problems (if it exists) must be an $N$-ball. We will also show
that stationary points of some functionals of a domain are balls.

\section{Preliminaries}

\begin{definition}
Let $K_{1}$ and $K_{2}$ be two compact subsets of $D.$ We call a
Hausdorff distance of $K_{1}$ and $K_{2}$ (or briefly
$d_{H}(K_{1},K_{2})),$ the following positive number:
\begin{equation*}
d_{H}(K_{1},K_{2})=\max \left[ \rho (K_{1},K_{2}),\rho
(K_{2},K_{1})\right] ,
\end{equation*}
where $\rho (K_{i},K_{j})={\max_{x\in K_{i}}}d(x,K_{j})$,
$i,j=1,2$\,, and $d(x,K_{j})={\min_{y\in K_{j}}}\left|
x-y\right|$\,.
\end{definition}

\begin{definition} \label{def:2.2}
Let $\omega _{n}$ be a sequence of open subsets of $D$ and let
$\omega $ be an open subset of $D$. Let $K_{n}$ and $K$ be their
complements in $\bar{D}$. We say that the sequence $\omega _{n}$
converges in the Hausdorff sense, to $\omega $ (or briefly $\omega
_{n}\overset{H}{\longrightarrow }\omega )$ if
\begin{equation*}
{\lim_{n\rightarrow +\infty}}d_{H}(K_{n},K)=0.
\end{equation*}
\end{definition}

\begin{definition}
Let $\{\omega _{n},\omega\}$ be a sequence of open subsets of $D$.
We say that the sequence $\omega _{n}$ converges in the compact
sense, to $\omega $\ (or briefly $\omega
_{n}\overset{K}{\longrightarrow }\omega $) if

\begin{itemize}
\item  every compact subset of $\omega$ is included in $\omega
_{n}$, for $n$ \textit{large enough}, and

\item  every compact subset of $\bar{\omega }^{c}$ is included in
$\overline{\omega }_{n}^{c}$, for $n$ \textit{large enough.}
\end{itemize}
\end{definition}

\begin{definition}
 Let $\{\omega _{n},\omega\}$ be a sequence of open subsets of
$D$. We say that the sequence $\omega _{n}$ converges in the sense
of characteristic functions, to $\omega $ (or briefly $\omega _{n}
\overset{L}{\longrightarrow }\omega $) if $\chi _{\omega _{n}}$
converges to $\chi _{\omega }$ in
$L_{\text{loc}}^{p}(\mathbb{R}^{N})$, $p\neq \infty$\,, ($\chi
_{\omega }$ is the characteristic function of $\omega $).
\end{definition}

\begin{definition} \cite{ba1}
Let $C$ be a compact convex set, the bounded domain
$\omega$ satisfies $C$-\textsc{gnp} if
\begin{enumerate}
\item  $\omega \supset \text{int}(C)$,

\item  $\partial \omega \setminus C$ is locally Lipschitz,

\item  for any $c\in \partial C$ there is an outward normal ray
$\Delta _{c}$ such that $\Delta _{c}\cap \omega $ is connected,
and

\item  for every $x\in \partial \omega \setminus C$ the inward
normal ray to $\omega $ (if exists) meets $C$.
\end{enumerate}
\end{definition}

\begin{remark}
If $\Omega $ satisfies the $C$-\textsc{gnp} and $C$ has a nonempty
interior, then $\Omega $ is connected.
\end{remark}
Put
\begin{equation*}
\mathcal{O}_{C}=\left\{ \omega \subset D: \omega\quad
\text{satisfies }C-\textsc{gnp}\right\}.
\end{equation*}
\begin{theorem}
If $\omega _{n}\in \mathcal{O}_{C}$\,, then there exist an open
subset $\omega \subset D$ and a subsequence (again denoted by
$\omega _{n}$) such that (i) $\omega
_{n}\overset{H}{\longrightarrow }\omega $\,, (ii) $\omega
_{n}\overset{K}{\longrightarrow }\omega$\,, (iii) $\chi _{\omega
_{n}}$ converges to\ $\chi _{\omega }$\ in $L^{1}(D)$ and (iv)
$\omega\in \mathcal{O}_{C}$. Furthermore, the assertions (i), (ii)
and (iii) are equivalent.
\end{theorem}

Barkatou proved this theorem \cite[Theorem 3.1]{ba1} and the
equivalence between (i), (ii) and (iii) \cite[Propositions 3.4,
3.5, 3.6, 3.7 and 3.8]{ba1}.

\begin{proposition}
Let $\left\{ \omega _{n},\omega \right\}\subset \mathcal{O}_{C}$
such that $\omega _{n}\overset{H}{\longrightarrow }\omega $\,. Let
$u_{n}$ and $u_{\omega }$ be respectively the solutions of
$P(\omega _{n})$ and $P(\omega)$\,. Then $u_{n}$ converges
strongly in $H_{0}^{1}(D)$ to $u_{\omega }$ ($u_{n}$ and
$u_{\omega }$ are extended by zero in $D$).
\end{proposition}

This proposition was proven for $N=2$ or $3$ \cite[Theorem
4.3]{ba1}).

\begin{definition}
Let $C$ be a convex set. We say that an open subset $\omega $ has
the $C$-\textsc{sp}, if

\begin{enumerate}
\item  $\omega \supset \text{int}(C)$\,,

\item  $\partial \omega \setminus C$ is locally Lipschitz,

\item  for any $c\in \partial C$ there is an outward normal ray
$\Delta _{c}$ such that $\Delta _{c}\cap \omega $ is connected,
and

\item  for all $x\in \partial \omega \setminus C \quad K_{x}\cap
\omega =\emptyset$\,, where $K_{x}$ is the closed cone defined by
\begin{equation*}
\left\{ y\in \mathbb{R}^{N} : (y-x).(z-x)\leq 0\,, \text{ for all
} z\in C\right\}.
\end{equation*}
\end{enumerate}
\end{definition}

\begin{remark}
$K_{x}$ is the normal cone to the convex hull of $C$ and $\{x\}$.
\end{remark}

\begin{proposition}\cite[Proposition 2.3]{ba1}
$\omega$ has the $C$-\textsc{gnp} if and only if $\omega $
satisfies the $C$-\textsc{sp}.
\end{proposition}

\begin{proposition}\cite[Theorem 3.5]{bt}
Let $v_{n}$ and $\;v_{\omega }$ be respectively the solutions of
the Dirichlet problems $P(\omega _{n},g_{n})$ and $P(\omega ,g )$.
If $g_{n}$ converges strongly in $H^{-1}(D)$ to $g$ then $v_{n}$
converges strongly in $H_{0}^{1}(D)$ to $v_{\omega }$ ($v_{n}$ and
$v_{\omega }$ are extended by zero in $D$).
\end{proposition}

\begin{lemma}\cite{bz,pi}
Let $\omega_{n}$ be a sequence of open and bounded subsets of $D$.
There exist a subsequence (again denoted by $\omega_{n}$) and some
open subset $\omega$ of $D$ such that
\begin{enumerate}
    \item $\omega_{n}$ converges to $\omega$ in the Hausdorff
    sense, and
    \item
    $|\omega|\leq\liminf_{n\rightarrow\infty}|\omega_{n}|$.
\end{enumerate}
\end{lemma}
\begin{theorem}
$\mathcal{QS}(f,k)$ has a solution if and only if $\int_{C}fdx>k|\partial C|$.
\end{theorem}
\begin{lemma}
Let $u_{\Omega}\in C^{2,\alpha}(\bar\Omega)$ be the solution of $P(\Omega,N)$. Then $\Omega$ is an $N$-ball if and only if
  $|\nabla u_{\Omega}(x)|=\frac{1}{H_{\partial\Omega}(x)},$ for every $x\in\partial\Omega$.
\end{lemma}
For the proof of this lemma, see Theorem 2.4 \cite{mp}.
\begin{theorem}
Let $u_{\Omega}$ (respectively $v_{\Omega}$) be the solution of $P(\Omega,1)$ (respectively $P(\Omega,u_{\Omega})$). If one of the following conditions is satisfied, then $\Omega$ is an $N$-ball.
\begin{enumerate}
  \item $|\nabla v_{\Omega}|=c\text{ on }\partial\Omega.$
  \item $|\nabla v_{\Omega}|=cx.\nu\text{ on }\partial\Omega.$
  \item $|\nabla v_{\Omega}|=c|\nabla u_{\Omega}|\text{ on }\partial\Omega.$
\end{enumerate}
\end{theorem}
For the proof of this theorem, see \cite{ps}.

As we use the standard tool of the domain derivative \cite{sz} to prove many of the propositions we state here, we recall its definition.\\

Suppose that the open $\omega$ is of class $C^{2}$. Consider a deformation field $V\in C^{2}(\mathbb{R}^{N};\mathbb{R}^{N})$ and set $\omega_{t}=\omega+tV(\omega)$, $t>0$. The application $Id+tV$ (a perturbation of the identity) is a Lipschitz diffeomorphism for $t$ sufficiently small and, by definition, the derivative of $J$ at $\omega$ in the direction $V$ is
$$dJ(\omega,V)=\lim_{t\rightarrow 0}\frac{J(\omega_{t})-J(\omega)}{t}.$$
As the functional $J$ depends on the domain $\omega$ through the solution of some Dirichlet problem, we need to define the domain derivative $u_{\omega}^{'}$ of $u_{\omega}$:
$$u_{\omega}^{'}=\lim_{t\rightarrow 0}\frac{u_{\omega_{t}}-u_{\omega}}{t}.$$
Furthermore, $u_{\omega}^{'}$ is the solution of the following problem:
\begin{equation}
 \left\{
\begin{array}{c}
-\Delta u'_{\omega}=0 \quad \text{in  }\omega\\
u_{\omega}^{'}=-\frac{\partial u_{\omega}}{\partial\nu}V.\nu\text{ on  }\partial \omega.
\end{array}
\right.
\end{equation}
The domain derivative $v_{\omega}^{'}$ of $v_{\omega}$ (solution of $P(\omega,u_{\omega})$) is the solution of:
\begin{equation}
 \left\{
\begin{array}{c}
-\Delta v'_{\omega}=u_{\omega}^{'}\quad \text{in  }\omega\\
u_{\omega}^{'}=-\frac{\partial u_{\omega}}{\partial\nu}V.\nu\text{ on  }\partial \omega.
\end{array}
\right.
\end{equation}
Now, to compute the derivative of the functionals we consider below, recall the following:
\begin{enumerate}
  \item The domain derivative of the volume is $$\int_{\partial\omega}V.\nu d\sigma.$$
  \item The domain derivative of the perimeter is $$\int_{\partial\omega}(N-1)H_{\omega}V.\nu d\sigma,$$
  $H_{\omega}$ being the mean curvature of $\partial\omega$.
  \item Suppose that $u_{\omega}\in H^{1}_{0}(D)$ and $\omega$ is of class $C^{2}$, then
  \begin{itemize}
    \item[(a)] If $F(\omega)=\int_{\omega}u_{\omega}^{2}dx$, then $$dF(\omega,V)=2\int_{\omega}u_{\omega}u_{\omega}^{'}dx.$$
    But $v_{\omega}\in H^{1}_{0}(D)$ and $-\Delta v_{\omega}=u_{\omega}$ in $\omega$, so by Green's formula we obtain
    $$dF(\omega,V)=2\int_{\partial\omega}|\nabla u_{\omega}||\nabla v_{\omega}|V.\nu d\sigma.$$
    \item[(b)] If $G(\omega)=\int_{\omega}|\nabla u_{\omega}|^{2}dx$, then by Hadamard's formula $$dG(\omega,V)=\int_{\partial\omega}|\nabla u_{\omega}|^{2}V.\nu d\sigma.$$
  \end{itemize}
\end{enumerate}
Since the set $\omega$ satisfies the $C$-\textsc{gnp}, we ask the
deformation set $\omega_{t}$ to satisfy the same property (for $t$
sufficiently small). The aim in the sequel is to prove that the $C$-\textsc{gnp} is stable by small deformation.\\
 $\omega$ having the
$C$-\textsc{gnp}, by Proposition 2.2, it satisfies the
$C$-\textsc{sp}. Then
\begin{equation*}
\text{for all }x\in \partial \omega \setminus C:K_{x}\cap \omega
=\emptyset.
\end{equation*}
For $t$ sufficiently small, let $\omega _{t}=\omega +tV\left(
\omega \right)$ be the deformation of $\omega $ in the direction
$V$. Let $x_{t}\in \partial \omega _{t}$. There exists $x\in
\partial \omega $ such that $x_{t}=x+tV(x)$. Using the definition
of $K_{x_{t}}$ and the equality above, we get (for $t$ small
enough and for every displacement $V$):
\begin{equation*}
\text{for all }x_{t}\in \partial \omega _{t}\setminus
C:K_{x_{t}}\cap \omega _{t}=\emptyset,
\end{equation*}
which means that $\omega _{t}$\ satisfies the $C$-\textsc{sp} (and
so the $C$-\textsc{gnp}) for every direction $V$\ when $t$ is
sufficiently small. In fact, suppose, by contradiction, there exists $x_{t}\in\partial\Omega_{t}\setminus C$ such that $K_{x_{t}}\cap\Omega_{t}\neq\emptyset$. Let $y_{t}\in K_{x_{t}}\cap\Omega_{t}$, there exists $y\in\Omega$, $y=y_{t}-tV(y)$ such that:
$$\forall c\in C,\;\;(y_{t}-x_{t}).(c-x_{t})\leq 0.$$
Show that $y\in K_{x}$:
\begin{eqnarray*}
    (y-x).(c-x)&=&(y_{t}-tV(y)-x_{t}+tV(x)).(c-x_{t}+tV(x))\\
    &=&(y_{t}-x_{t}+t(V(y)-V(x))).(c-x_{t}+tV(x))\\
    &=&(y_{t}-x_{t}).(c-x_{t})+\epsilon(t)
\end{eqnarray*}
where $\epsilon(t)=t(y_{t}-x_{t}).V(x)+t(V(y)-V(x)).(c-x_{t})+t^2(V(y)-V(x)).V(x)$ which, as $t$, tends to $0$. Obtaining the contradiction.

\section{Existence and symmetry for $\mathcal{QS}(f,k)$ and $\mathcal{P}(f,c)$}

Suppose there exists $\Omega$ a solution of $\mathcal{P}(f,c)$. Then by Cauchy-Schwarz's inequality
$$\sqrt{c}|\partial\Omega|=\int_{\partial\Omega}\sqrt{|\nabla u_{\Omega}||\nabla v_{\Omega}|}\leq (\int_{\partial\Omega}|\nabla u_{\Omega}|)^{\frac{1}{2}}(\int_{\partial\Omega}|\nabla v_{\Omega}|)^{\frac{1}{2}}.$$
And by Green's formula
$$c|\partial\Omega|^{2}\leq (\int_{\Omega}f)(\int_{\Omega}u_{\Omega}).$$
Now $C\subset\Omega$ and $C$ is convex, then $|\partial C|<|\partial\Omega|$, and so,
$$c|\partial C|^{2}<\int_{C}f\int_{\Omega}u_{\Omega}.$$

In the sequel, we will prove
\begin{theorem}
If $c|\partial C|^{2}<\int_{C}f\int_{C}u_{C}$, then there exists $\Omega\varsupsetneq C$ solution of $\mathcal{P}(f,c)$.
\end{theorem}
To prove this theorem, we proceed as follows.\\
By using the domain derivative \cite{sz}, the problem $\mathcal{P}(f,c)$
seems to be the Euler equation of the following optimization
problem. Put
$$\mathcal{O}_{C}=\left\{ \omega \subset D: \omega\quad
\text{satisfies }C-\textsc{gnp}\right\}.$$
Find $\Omega\in\mathcal{O}_{C}$ such that
$$J(\Omega)=\text{Min}\{J(\omega),\;\omega\in\mathcal{O}_{C}\},$$
and
$$J(\omega)=c|\omega|-\frac{1}{2}\int_{\omega}u^{2}_{\omega}.$$
$u_{\omega}$ is the solution of $P(\omega,f)$.

\begin{proposition}
\begin{enumerate}
    \item There exists $\Omega\in\mathcal{O}_{C}$ such that
    $$J(\Omega)=\text{Min}\{J(\omega),\;\omega\in\mathcal{O}_{C}\}.$$
    \item If $\Omega$ is of class $C^2$, then
    \begin{equation*}
(I) \left\{
\begin{array}{c}
|\nabla u_{\Omega}||\nabla v_{\Omega}|\leq c \;\text{on}\;\partial\Omega\cap\partial C\\
|\nabla u_{\Omega}||\nabla v_{\Omega}|= c\;\text{on}\;\partial\Omega\setminus\partial
C.
\end{array}
\right.
\end{equation*}
\end{enumerate}
\end{proposition}
Now, put

$$M_C=\frac{1}{|\partial C|^{2}}\int_{C}f(x)dx\int_{C}u_{C},$$
$$F(\omega)=M_{C}|\omega|-\frac{1}{2}\int_{\omega}u^{2}_{\omega},\;\; \text{and}$$
$$\mathcal{O}_{\Omega}=\{\omega\subset\Omega,\;\omega\in\mathcal{O}_{C}\},$$

\begin{proposition}
\begin{enumerate}
    \item There exists $\Omega^*\in\mathcal{O}_{\Omega}$ such that
    $$F(\Omega)=\text{Min}\{F(\omega),\;\omega\in\mathcal{O}_{\Omega}\}.$$
    \item If $\Omega^*$ is of class $C^2$, then
    \begin{equation*}
(II) \left\{
\begin{array}{c}
|\nabla u_{\Omega^*}||\nabla v_{\Omega^*}|\leq M_C\;\text{on}\;\partial\Omega^*\cap\partial C\\
|\nabla u_{\Omega^*}||\nabla v_{\Omega^*}|\geq M_C \;\text{on}\;\partial\Omega^*\cap\partial\Omega\\
|\nabla u_{\Omega^*}||\nabla v_{\Omega^*}|=
M_C\;\text{on}\;\partial\Omega^*\setminus(\partial
C\cup\partial\Omega).
\end{array}
\right.
\end{equation*}
\end{enumerate}
\end{proposition}
The proof of the above propositions uses Theorem 2.1, Proposition 2.1, Proposition 2.3 and Lemma 2.1.\\
Next, we prove by contradiction that $\partial\Omega\cap\partial C=\emptyset$. The contradiction is obtained according to $(I)$,
$(II)$ and by applying the maximum principle to $\Omega$ and
$\Omega^*$.\\
Since
$\text{int}(C)\subset\Omega^*\subset\Omega$, one of the following
situations occurs.
\begin{enumerate}
    \item $\partial\Omega\equiv\partial C$
    \item $\partial\Omega\neq\partial C$ and $\partial\Omega^*\equiv\partial C$
    \item $\partial\Omega\neq\partial C$ and $\partial\Omega^*\neq\partial C$
    \item $\partial\Omega\neq\partial C$ and $\partial\Omega\equiv\partial\Omega^*$
    \item $\partial\Omega\neq\partial C$ and $\partial\Omega\neq\partial\Omega^*$
\end{enumerate}
For any of the above cases, using the maximum principle together with $(I)$ and
$(II)$, we obtain

$$c<M_C=|\nabla u_{\Omega^*}||\nabla v_{\Omega^*}|\leq |\nabla u_{\Omega}||\nabla v_{\Omega}|\leq c,\; \text{on}\; \partial\Omega^{*}\cap\partial\Omega\cap\partial C,$$
which is absurd.\\

Up to now, we will investigate several situations where an integral inequality on some domain $\Omega$ says that
\begin{enumerate}
  \item either $\Omega$ is (strictly) contained in the solution of $\mathcal{QS}(f,\mu)$ (or of $\mathcal{P}(f,\mu)$) (for some constant $\mu$),
  \item or, $\Omega$ is an $N$-ball.
\end{enumerate}

In the sequel, we suppose that $int(C)\subset\Omega\subset D$.
\begin{proposition}
Let $u_{\Omega}$ be the solution of $P(\Omega,f)$. If $$|\nabla u_{\Omega}|\geq k\;\text{ on }\partial
 \Omega,$$ then
\begin{enumerate}
  \item either, there exists $\Omega^{*}\varsupsetneq\bar\Omega$ such that $\Omega^{*}$ is solution to $\mathcal{QS}(f,k)$,
  \item or, $\int_{C}f=k|\partial\Omega|$, in that case,
  \begin{itemize}
    \item[(i)] either, there exists $\Omega^{**}\varsupsetneq C$ such that $\Omega^{**}$ is solution to $\mathcal{QS}(f,k)$
    \item[(ii)] or, $\Omega=int(C)$ and $|\nabla u_{\Omega}|=k\;\text{ on }\partial \Omega$.
  \end{itemize}
\end{enumerate}
\end{proposition}
\begin{proof}
By Green's formula,
$$\int_{\Omega}f=\int_{\partial\Omega}|\nabla u_{\Omega}|\geq k|\partial\Omega|.$$
Then
\begin{enumerate}
  \item either $\int_{\Omega}f> k|\partial\Omega|$, and so, by Theorem 2.2, there exists
  $\Omega^{*}\varsupsetneq\bar\Omega$ such that $\Omega^{*}$ is solution to $\mathcal{QS}(f,k)$.
  \item or, $\int_{\Omega}f= k|\partial\Omega|$, in that case, since $C$ is convex and $int(C)\subset \Omega$, then
  $\int_{C}f\geq k|\partial C|$, and so
  \begin{enumerate}
    \item either $\int_{C}f> k|\partial C|$, and so,
  $\Omega^{**}\varsupsetneq C$ such that $\Omega^{**}$ is solution to $\mathcal{QS}(f,k)$.
    \item or, $\int_{C}f= k|\partial C|$, and so $\Omega=int(C)$, $|\nabla u_{\Omega}|=k\;\text{ on }\partial \Omega$.
  \end{enumerate}
\end{enumerate}
\end{proof}

\begin{corollary}
Let $u_{\Omega}$ be the solution of $P(\Omega,\chi_{C})$. Suppose that $$|\nabla u_{\Omega}|\geq k\;\text{ on }\partial
 \Omega$$ Then
\begin{enumerate}
  \item either, there exists $\Omega^{*}\varsupsetneq\bar\Omega$ such that $\Omega^{*}$ is solution to $\mathcal{QS}(\chi_{C},k)$,
  \item or, $|C|=k|\partial\Omega|$, in that case
  \begin{itemize}
    \item[(i)] either, there exists $\Omega^{**}\varsupsetneq C$ such that $\Omega^{**}$ is solution to $\mathcal{QS}(\chi_{C},k)$.
    \item[(ii)] or, $\Omega$ is an $N$-ball.
  \end{itemize}
\end{enumerate}
\end{corollary}

\begin{proof}
If we replace in the previous proof $f$ by $\chi_{C}$, we obtain item 1. and item 2 (i).
For item 2 (ii),  we obtain $\int_{C}f= k|\partial C|$, and $|\nabla u_{\Omega}|=k\;\text{ on }\partial \Omega$ which means that $\Omega$
is solution to Serrin's problem and so it coincides with an $N$-ball.
\end{proof}

\begin{corollary}
Let $u_{C}$ be the solution of $P(C,1)$ and let $v_{\Omega}$ be the solution of $P(\Omega,u_{C})$. Suppose that
$$|\nabla v_{\Omega}|\geq k\;\text{ on }\partial\Omega$$ Then
\begin{enumerate}
  \item either, there exists $\Omega^{*}\varsupsetneq\bar\Omega$ such that $\Omega^{*}$ is solution to $\mathcal{QS}(u_{C},k)$,
  \item or, $\int_{C}u_{C}=k|\partial\Omega|$, in that case
  \begin{itemize}
    \item[(i)] either, there exists $\Omega^{**}\varsupsetneq C$ such that $\Omega^{**}$ is solution to $\mathcal{QS}(u_{C},k)$,
    \item[(ii)] or, $\Omega$ is an $N$-ball.
  \end{itemize}
\end{enumerate}
\end{corollary}

\begin{proof}
If we replace in the previous proof $f$ by $u_{C}$, we obtain obtain item 1. and item 2 (i).
For item 2 (ii),  we obtain $\int_{C}f= k|\partial C|$, $u_{\Omega}=u_{C}$, $v_{\Omega}=v_{C}$ and $|\nabla v_{\Omega}|=k\;\text{ on }\partial \Omega$. But according to Theorem 2.3, this means that $\Omega$ is an $N$-ball.
\end{proof}

\begin{proposition}
Let $u_{\Omega}$ be the solution of $P(\Omega,f)$. Suppose that $$|\nabla u_{\Omega}|\geq k\;\text{ on }\partial \Omega$$ Then
\begin{enumerate}
  \item either, there exists $\Omega^{*}\varsupsetneq\bar\Omega$ such that $\Omega^{*}$ is solution to $\mathcal{QS}(f^{2},k^{2}\frac{|\partial\Omega|}{|\Omega|})$.
  \item or, $\Omega$ is an $N$-ball.
\end{enumerate}
\end{proposition}

\begin{proof}
By Green's formula and Cauchy-Schwarz inequality,
$$k|\partial\Omega|\leq \int_{\partial\Omega}|\nabla u_{\Omega}|=\int_{\Omega}f\leq |\Omega|^{\frac{1}{2}}(\int_{\Omega}f^{2})^{\frac{1}{2}}.$$
Then $$k^{2}|\partial\Omega|^{2}\leq|\Omega|\int_{\Omega}f^{2}.$$
As in the proof of Proposition 3.1, we get item 1. For item 2., $k^{2}|\partial\Omega|^{2}=|\Omega|\int_{\Omega}f^{2}$, implies that $f$ is constant in $\Omega$, $|\nabla u_{\Omega}|=k\;\text{ on }\partial \Omega$ which means that $\Omega$ is solution to Serrin's problem and so $\Omega$ is an $N$-ball.

\end{proof}

\begin{remark}
Put $f\equiv 1$. Let $u_{\Omega}$ (respectively $v_{\Omega}$) be the solution of $P(\Omega,1)$ (respectively $P(\Omega,u_{\Omega})$). Suppose $$|\nabla v_{\Omega}|\geq k\;\text{ on }\partial \Omega.$$
Since $u_{\Omega}$ is not constant in $\Omega$, then there exists $\Omega^{*}\varsupsetneq\bar\Omega$ such that $\Omega^{*}$ is solution to $\mathcal{QS}(u_{\Omega}^{2},k^{2}\frac{|\partial\Omega|}{|\Omega|})$.
\end{remark}

\begin{remark}
When $\Omega$ satisfies the uniform interior ball property with radius $R$, one can replace in the preceding propositions the constant $k$ by $NR$.
\end{remark}

\begin{proposition}
Let $u_{\Omega}$ be the solution of $P(\Omega,f)$ and let $v_{\Omega}$ be the solution of $P(\Omega,u_{\Omega})$. Suppose that
$$|\nabla u_{\Omega}||\nabla v_{\Omega}|\geq c\text{ on }\partial \Omega$$ Then,
\begin{enumerate}
  \item either, there exists $\Omega^{*}\varsupsetneq\bar\Omega$ such that $\Omega^{*}$ is solution to $\mathcal{P}(f,c)$,
  \item or, $|\nabla u_{\Omega}||\nabla v_{\Omega}|= c\text{ on }\partial \Omega$.
\end{enumerate}
\end{proposition}
\begin{proof}
By Green's formula and Cauchy-Schwarz inequality and Green's formula
$$\sqrt{c}|\partial\Omega|=\int_{\partial\Omega}\sqrt{|\nabla u_{\Omega}||\nabla v_{\Omega}|}\leq (\int_{\partial\Omega}|\nabla u_{\Omega}|)^{\frac{1}{2}}(\int_{\partial\Omega}|\nabla v_{\Omega}|)^{\frac{1}{2}}\leq (\int_{\Omega}f)^{\frac{1}{2}}(\int_{\Omega}u_{\Omega})^{\frac{1}{2}}.$$
So
$$c|\partial\Omega|^{2}\leq (\int_{\Omega}f)(\int_{\Omega}u_{\Omega}).$$
This inequality allows us to get
\begin{enumerate}
  \item either, there exists $\Omega^{*}\varsupsetneq\bar\Omega$ such that $\Omega^{*}$ is solution to $\mathcal{P}(f,c)$,
  \item or, $|\nabla u_{\Omega}||\nabla v_{\Omega}|= c\text{ on }\partial \Omega$.
\end{enumerate}

\end{proof}
\begin{remark}
Item 1. of the previous proposition implies the existence of
 $\Omega^{1}$ (respectively $\Omega^{2}$) which (strictly) contains $\bar\Omega$ and such that $\Omega^{1}$ is solution to $\mathcal{QS}(u_{\Omega},c\frac{|\partial\Omega|}{\int_{\Omega}f})$ (respectively $\Omega^{2}$ is solution to $\mathcal{QS}(f,c\frac{|\partial\Omega|}{\int_{\Omega}u_{\Omega}})$),
 \end{remark}

\begin{proposition}
Let $u_{\Omega}$ be the solution of $P(\Omega,1)$ and let $v_{\Omega}$ be the solution of $P(\Omega,u_{\Omega})$. Suppose that $$|\nabla u_{\Omega}||\nabla v_{\Omega}|\geq c\text{ on }\partial \Omega$$ Then
\begin{enumerate}
  \item either, there exists $\Omega^{*}\varsupsetneq\bar\Omega$ such that $\Omega^{*}$ is solution to $\mathcal{QS}(u_{\Omega},c\frac{|\partial\Omega|}{|\Omega|})$,
  \item or, $\Omega$ is an $N$-ball.
  \end{enumerate}
\end{proposition}
\begin{proof}
When $f\equiv 1$, according to the previous proposition,
 \begin{enumerate}
  \item either, there exists $\Omega^{*}\varsupsetneq\bar\Omega$ such that $\Omega^{*}$ is solution to $\mathcal{QS}(u_{\Omega},c\frac{|\partial\Omega|}{|\Omega|})$,
  \item or, $|\nabla u_{\Omega}||\nabla v_{\Omega}|= c\text{ on }\partial \Omega$, and so $\Omega$ is an $N$-ball according to Theorem 2.3.
\end{enumerate}
Remark that, we also get $|\nabla v_{\Omega}|=\lambda|\nabla u_{\Omega}|\text{ on }\partial\Omega$  (where $\lambda$ is constant). This implies that, in particular, that $\Omega$ is solution to Serrin's problem with $\sqrt{\lambda c}$ as constant.
\end{proof}

\begin{remark}
When $\Omega$ satisfies the uniform interior ball property with radius $R$, one can replace in the preceding propositions the constant $c$ by $\frac{N^{2}}{N+2}R^{4}$.
\end{remark}

\begin{proposition}
Let $u_{\Omega}$ be the solution of $P(\Omega,1)$ and let $v_{\Omega}$ be the solution of $P(\Omega,u_{\Omega})$. Suppose that
$$|\nabla v_{\Omega}|\geq k|\nabla u_{\Omega}|\text{ on }\partial \Omega$$ Then
\begin{enumerate}
  \item either, there exists $\Omega^{*}\varsupsetneq\bar\Omega$ such that $\Omega^{*}$ is solution to $\mathcal{QS}(u_{\Omega},k\frac{|\Omega|}{|\partial\Omega|})$,
  \item or, $\Omega$ is an $N$-ball.
\end{enumerate}
\end{proposition}

\begin{proof}
By Green's formula,
$$k\frac{|\Omega|}{|\partial\Omega|}|\partial\Omega|=k\int_{\partial\Omega}|\nabla u_{\Omega}|\leq\int_{\partial\Omega}|\nabla v_{\Omega}|=\int_{\Omega}u_{\Omega}.$$
Then
\begin{enumerate}
  \item either, there exists $\Omega^{*}\varsupsetneq\bar\Omega$ such that $\Omega^{*}$ is solution to $\mathcal{QS}(u_{\Omega},k\frac{|\Omega|}{|\partial\Omega|})$,
  \item or, $|\nabla v_{\Omega}|=k|\nabla u_{\Omega}|\text{ on }\partial \Omega$ and so $\Omega$ is an $N$-ball.
\end{enumerate}

\end{proof}

\begin{remark}
When $\Omega$ satisfies the uniform interior ball property with radius $R$, one can replace in the preceding propositions the constant $k$ by $\frac{N}{N+2}R^{3}$.
\end{remark}
\begin{proposition}
Let $u_{\Omega}$ be the solution of $P(\Omega,1)$ and let $v_{\Omega}$ be the solution of $P(\Omega,u_{\Omega})$. Suppose that
$$|\nabla v_{\Omega}|\geq kx.\nu\text{ on }\partial \Omega.$$ Suppose that $\Omega$ is starshaped w.r.t. the origin, then
\begin{enumerate}
  \item either, there exists $\Omega^{*}\varsupsetneq\bar\Omega$ such that $\Omega^{*}$ is solution to $\mathcal{QS}(u_{\Omega},kN\frac{|\Omega|}{|\partial\Omega|})$,
  \item or, $\Omega$ is an $N$-ball.
\end{enumerate}
\end{proposition}
\begin{proof}
Since $\Omega$ is starshaped w.r.t. the origin,
$$kN\frac{|\Omega|}{|\partial\Omega|}|\partial\Omega|=k\int_{\partial\Omega}x.\nu\leq\int_{\partial\Omega}|\nabla v_{\Omega}|=\int_{\Omega}u_{\Omega}.$$
Then
\begin{enumerate}
  \item either, there exists $\Omega^{*}\varsupsetneq\bar\Omega$ such that $\Omega^{*}$ is solution to $\mathcal{QS}(u_{\Omega},kN\frac{|\Omega|}{|\partial\Omega|})$,
  \item or, $|\nabla v_{\Omega}|=kx.\nu\text{ on }\partial \Omega$ and so $\Omega$ is an $N$-ball.
\end{enumerate}
According to Theorem 2.3, $\Omega$ is an $N$-ball.
\end{proof}
\section{Symmetry results for some overdetermined problems}
In this section, we consider new overdetermined boundary value problems. We use essentially the domain derivative to get the symmetry result.
\begin{proposition}
Let $\Omega\subset D$, and let $u_{\Omega}$ and $v_{\Omega}$ be the solutions of the Dirichlet problems $P(\Omega,1)$ and $P(\Omega,u_{\Omega})$. Suppose that
\begin{equation*}
(OC)\left\{
\begin{array}{c}
|\nabla v_{\Omega}|=\frac{N}{N+2}|\nabla u_{\Omega}|^{3} \text{ on }\partial\Omega,\;\text{and} \\
\int_{\partial\Omega}|\nabla u_{\Omega}|^{2}\frac{\partial u_{\Omega}^{'}}{\partial\nu}d\sigma=0.
\end{array}
\right.
\end{equation*}
Then $\Omega$ is an $N$-ball.
\end{proposition}

Before proving this proposition, we check that a ball $B_{R}$ is a solution to problem $\mathcal{P}(1,c)$ which satisfies $(OC)$.\\
Let $u_{R}$ be the solution of $P(B_{R},1)$. Using polar
coordinates, $u_{R}$ verifies
\begin{eqnarray*}
-u''_R-\frac{N-1}{r}u'_{R}
&=&1\quad\text{for \ }r\in \left] 0,R\right[ , \\
u_R(R) &=&0.
\end{eqnarray*}
By the first equation, $\left( r^{N-1}u'_{R}\right)'=-r^{N-1}$.
Since $u_{R}(R)=0$, we get
\begin{equation*}
r^{N-1}u'_R\left( r\right) =R^{N-1}u'_{R}\left( R\right)
+\int_{r}^{R}s^{N-1}ds.
\end{equation*}
As $r\rightarrow 0$, $r^{N-1}u'_{R}\left( r\right) \rightarrow 0$
(otherwise we get a distributional contribution to $\Delta u_{R}$
at the origin). Thus
\begin{equation*}
-u'_R\left(R\right)
=\frac{1}{R^{N-1}}\int_{0}^{R}s^{N-1}ds=\frac{R}{N}.
\end{equation*}

Now let $v_{R}$) be the solution of $P(B_{R},u_{R})$). On one hand, we have $ -u_{R}^{\prime
}(R)=\frac{R}{N}.$ Then a simple calculation shows that
$$u_{R}(r)=\frac{1}{2N}\left( R^{2}-r^{2}\right)\quad\text{for \ }r\in \left] 0,R\right[.$$
On the other hand, the radial function $v_{R}$
satisfies
\begin{equation*}
\left\{
\begin{array}{c}
-v_{R}^{^{\prime \prime }}-\frac{N-1}{r}v_{R}^{^{\prime
}}=u_{R}\quad\text{for }r\in
\left] 0,R\right[ , \\
v_{R}(R)=0 \\
-v_{R}^{^{\prime }}\left( R\right)
=\frac{1}{R^{N-1}}\int_{0}^{R}s^{N-1}u_{R}(s)ds=
\frac{N}{N+2}\left( \frac{R}{N}\right) ^{3}.
\end{array}
\right.
\end{equation*}
Therefore $B_{R}$ is a solution of $\mathcal{P}(1,c)$ and so
$$|\nabla v_{R}|=\frac{N}{N+2}\left( \frac{R}{N}\right) ^{3}=\frac{N}{N+2}|\nabla u_{R}| ^{3},\;\text{and}$$
$$\int_{\partial\Omega}|\nabla u_{\omega}|^{2}\frac{\partial u_{\Omega}^{'}}{\partial\nu}d\sigma=(\frac{R}{N})^{2}\int_{\partial\Omega}\frac{\partial u_{\Omega}^{'}}{\partial\nu}d\sigma=0.$$

\begin{proof}
Let $\mathcal{O}$ be the set of all bounded open connected domains of class $C^{2}$ in $\mathbb{R}^N$. Consider the following functional introduced in \cite{hc}:
$$J(\omega)=N\int_{\partial\omega}|\nabla u_{\omega}|^{3}d\sigma-(N+2)\int_{\omega}|\nabla u_{\omega}|^{2}dx.$$
$u_{\omega}$ being the solution of $P(\omega,1)$. As in \cite{hc}, $J(\omega)\geq 0$ for any $\omega\in \mathcal{O}$. Now, by Green's formula $$\frac{N}{N+2}\int_{\partial\Omega}|\nabla u_{\Omega}|^{3}=\int_{\partial\Omega}|\nabla v_{\Omega}|=\int_{\Omega}u_{\Omega}=\int_{\Omega}|\nabla u_{\Omega}|^{2},$$
then $J(\Omega)=0.$\\
The domain derivative of the functional $J$ at $\omega\in \mathcal{O}$ in the direction $V$ is given by
$$\displaystyle dJ(\omega,V)=\int_{\partial\omega}([2(N-1)|\nabla u_{\omega}|^{2}-2N(N-1)H_{\omega}|\nabla u_{\omega}|^{3}]V.\nu-3N|\nabla u_{\omega}|^{2}\frac{\partial u_{\omega}^{'}}{\partial\nu})d\sigma.$$
$H_{\omega}$ is the mean curvature of $\partial\omega$ and $u'_{\omega}$ is the domain derivative of $u_{\omega}$.
Since $\Omega$ minimizes the functional $J$, then for every vector field $V\in C^{2}(\mathbb{R}^{N},\mathbb{R}^{N})$ we have $dJ(\Omega,V)=0$. Now, according to $(OC)$, we obtain $$|\nabla u_{\Omega}|=\frac{1}{NH_{\Omega}}\text{ on }\partial\Omega.$$
This means that $\Omega$ is an $N$-ball, according to Theorem 2.2.
\end{proof}
\begin{remark}
Proposition 4.1  corresponds to the following biharmonic problem
\begin{equation*}
\left\{
\begin{array}{c}
\Delta^{2} v_{\Omega}=1 \quad \text{ in  }\Omega\,,
 v_{\Omega}=\Delta v_{\Omega}=0\quad (N+2)|\nabla v_{\Omega}|=N|\nabla u_{\Omega}|^{3}
 \text{ on  }\partial \Omega\,,\\
\int_{\partial\Omega}|\nabla u_{\Omega}|^{2}\frac{\partial u'_{\Omega}}{\partial\nu}=0\text{ on  }\partial\Omega.
\end{array}
\right.
\end{equation*}
\end{remark}
The work bellow is motivated by the following result obtained by Didenko and Emamizadeh \cite{de}
\begin{theorem}
The domain $\Omega$ is a ball if and only if there exists a constant $c$ that the following integral equation is valid
$$\int_{\Omega}u'_{\Omega}dx=c\int_{\partial\Omega}u'_{\Omega}d\sigma,$$
for every vector field $V\in C^{2}(\mathbb{R}^{N},\mathbb{R}^{N})$ where $$\triangle u'_{\Omega}=0\text{ in }\Omega\quad\quad
u'_{\Omega}=-\frac{\partial u_{\Omega}}{\partial\nu}V.\nu$$
\end{theorem}

Let $u_{\Omega}$, $v_{\Omega}$, and $w_{\Omega}$ be respectively the solution of $P(\Omega,1)$, $P(\Omega,u_{\Omega})$, $P(\Omega,v_{\Omega})$.

\begin{proposition}
$\Omega$ is an $N$-ball if one of the following conditions is satisfied
\begin{itemize}
     \item[(i)] $-\frac{\partial v'_{\Omega}}{\partial\nu}=c^{2}V.\nu\text{ on  }\partial \Omega$
     \item[(ii)] $-\frac{\partial v'_{\Omega}}{\partial\nu}=c|\nabla u_{\Omega}|V.\nu\text{ on  }\partial \Omega$
     \item[(iii)] $-\frac{\partial v'_{\Omega}}{\partial\nu}=c^{2}|\nabla u_{\Omega}|^{2}V.\nu\text{ on  }\partial \Omega$
\end{itemize}

\end{proposition}
\begin{proof}

By Green's formula
\begin{eqnarray*}
   \int_{\partial\Omega}-\frac{\partial v'_{\Omega}}{\partial\nu}=\int_{\Omega}u'_{\Omega} &=&\int_{\Omega}-\Delta u_{\Omega}u'_{\Omega}\\
    &=&\int_{\Omega}-\Delta u'_{\Omega}u_{\Omega}+\int_{\partial\Omega}u_{\Omega}\frac{\partial u'_{\Omega}}{\partial\nu}-\int_{\partial\Omega}u'_{\Omega}\frac{\partial u_{\Omega}}{\partial\nu}\\
    &=&\int_{\partial\Omega}|\nabla u_{\Omega}|^{2}V.\nu.
\end{eqnarray*}
Then, we get
\begin{itemize}
     \item[(i)] $|\nabla u_{\Omega}|=c\text{ on  }\partial \Omega$
     \item[(ii)] $|\nabla u_{\Omega}|=c\text{ on  }\partial \Omega$
     \item[(iii)] $|\nabla v_{\Omega}|=\frac{1}{c}|\nabla u_{\Omega}|\text{ on  }\partial \Omega$
\end{itemize}
 Each of the above items tells us that $\Omega$ is an $N$-ball.
\end{proof}
\begin{remark}
Proposition 4.2 corresponds to the following biharmonic problem
\begin{equation*}
\left\{
\begin{array}{c}
\Delta^{2} v'_{\Omega}=0 \quad \text{ in  }\Omega\,,\\
 v'_{\Omega}=|\nabla v_{\Omega}|V.\nu,\quad \Delta v'_{\Omega}=|\nabla u_{\Omega}|V.\nu
 \text{ on  }\partial \Omega\,,\\
-\frac{\partial w'_{\Omega}}{\partial\nu}=\Lambda\text{ on  }\partial \Omega.

\end{array}
\right.
\end{equation*}
where $\Lambda=c^{2}V.\nu$, $\Lambda=c|\nabla u_{\Omega}|V.\nu$, or $\Lambda=c^{2}|\nabla u_{\Omega}|^{2}V.\nu$.
\end{remark}
\begin{proposition}
$\Omega$ is an $N$-ball if one of the following conditions is satisfied:
\begin{itemize}
     \item[(i)] $-\frac{\partial w'_{\Omega}}{\partial\nu}=cV.\nu\text{ on  }\partial \Omega$
     \item[(ii)] $-\frac{\partial w'_{\Omega}}{\partial\nu}=c|\nabla u_{\Omega}|V.\nu\text{ on  }\partial \Omega$
     \item[(iii)] $-\frac{\partial w'_{\Omega}}{\partial\nu}=c|\nabla u_{\Omega}|^{2}V.\nu\text{ on  }\partial \Omega$
     \item[(iv)] $-\frac{\partial w'_{\Omega}}{\partial\nu}=c|\nabla v_{\Omega}|V.\nu\text{ on  }\partial \Omega$
   \end{itemize}
 \end{proposition}
 \begin{proof}
By Green's formula
\begin{eqnarray*}
   \int_{\partial\Omega}-\frac{\partial w'_{\Omega}}{\partial\nu}=\int_{\Omega}v'_{\Omega} &=&\int_{\Omega}-\Delta u_{\Omega}u'_{\Omega}\\
    &=&\int_{\Omega}-\Delta v'_{\Omega}u_{\Omega}+\int_{\partial\Omega}u_{\Omega}\frac{\partial v'_{\Omega}}{\partial\nu}-\int_{\partial\Omega}v'_{\Omega}\frac{\partial u_{\Omega}}{\partial\nu}\\
    &=&\int_{\Omega}uu'_{\Omega}+\int_{\partial\Omega}|\nabla u_{\Omega}||\nabla v_{\Omega}|V.\nu.\\
    &=&2\int_{\partial\Omega}|\nabla u_{\Omega}||\nabla v_{\Omega}|V.\nu.
\end{eqnarray*}
Then, we get
\begin{itemize}
     \item[(i)] $|\nabla u_{\Omega}||\nabla v_{\Omega}|=\frac{c}{2}\text{ on  }\partial \Omega$
     \item[(ii)] $|\nabla v_{\Omega}|=\frac{c}{2}\text{ on  }\partial \Omega$
     \item[(iii)] $|\nabla v_{\Omega}|=\frac{c}{2}|\nabla u_{\Omega}|\text{ on  }\partial \Omega$
     \item[(iv)] $|\nabla u_{\Omega}|=\frac{c}{2}\text{ on  }\partial \Omega$
\end{itemize}
Each of the above items tells us that $\Omega$ is an $N$-ball.
\end{proof}
\begin{remark}
Proposition 4.3 corresponds to the following triharmonic problem
\begin{equation*}
\left\{
\begin{array}{c}
\Delta^{3} w'_{\Omega}=0 \quad \text{ in  }\Omega\,,\\
 w'_{\Omega}=|\nabla w_{\Omega}|V.\nu,\quad \Delta w'_{\Omega}=|\nabla v_{\Omega}|V.\nu,\quad
 \Delta^{2} w'_{\Omega}=|\nabla u_{\Omega}|V.\nu \text{  on  }\partial\Omega\,, \\
-\frac{\partial w'_{\Omega}}{\partial\nu}=\Lambda\text{ on  }\partial \Omega.

\end{array}
\right.
\end{equation*}
where $\Lambda=cV.\nu$, $\Lambda=c|\nabla u_{\Omega}|V.\nu$, $\Lambda=c|\nabla u_{\Omega}|^{2}V.\nu$ or $\Lambda=c|\nabla v_{\Omega}|V.\nu$.
\end{remark}

\section{Concluding remarks}
\begin{remark}
Let $z_{\Omega}$ be the solution of $P(\Omega,\frac{1}{2}u^{2}_{\Omega})$.
\begin{proposition}
$\Omega$ is an $N$-ball if one of the following conditions is satisfied
\begin{itemize}
     \item[(i)] $-\frac{\partial z'_{\Omega}}{\partial\nu}=cV.\nu\text{ on  }\partial \Omega$
    \item[(ii)] $-\frac{\partial z'_{\Omega}}{\partial\nu}=c|\nabla u_{\Omega}|V.\nu\text{ on  }\partial \Omega$
     \item[(iii)] $-\frac{\partial z'_{\Omega}}{\partial\nu}=c|\nabla u_{\Omega}|^{2}V.\nu\text{ on  }\partial \Omega$
     \item[(iv)] $-\frac{\partial z'_{\Omega}}{\partial\nu}=c|\nabla v_{\Omega}|V.\nu\text{ on  }\partial \Omega$
   \end{itemize}
\end{proposition}
\begin{proof}
By Green's formula
\begin{eqnarray*}
   \int_{\partial\Omega}-\frac{\partial z'_{\Omega}}{\partial\nu}=\int_{\Omega}u_{\Omega}u'_{\Omega} &=&\int_{\Omega}-\Delta v_{\Omega}u'_{\Omega}\\
    &=&\int_{\Omega}-\Delta u'_{\Omega}v_{\Omega}+\int_{\partial\Omega}v_{\Omega}\frac{\partial u'_{\Omega}}{\partial\nu}-\int_{\partial\Omega}u'_{\Omega}\frac{\partial v_{\Omega}}{\partial\nu}\\
    &=&\int_{\partial\Omega}|\nabla u_{\Omega}||\nabla v_{\Omega}|V.\nu.
\end{eqnarray*}
Then, we get
\begin{itemize}
     \item[(i)] $|\nabla u_{\Omega}||\nabla v_{\Omega}|=c\text{ on  }\partial \Omega$
     \item[(ii)] $|\nabla v_{\Omega}|=c\text{ on  }\partial \Omega$
     \item[(iii)] $|\nabla v_{\Omega}|=c|\nabla u_{\Omega}|\text{ on  }\partial \Omega$
     \item[(iv)] $|\nabla u_{\Omega}|=c\text{ on  }\partial \Omega$
\end{itemize}
Each of the above items tells us that $\Omega$ is an $N$-ball.
\end{proof}

\end{remark}

\begin{remark}
\begin{proposition}
 Balls are stationary for the functional $$J(\Omega)=\frac{F(\omega)}{G(\omega)},$$
where
 \begin{enumerate}
        \item $F(\omega)=|\omega|$ and $G(\omega)=\int_{\omega}u_{\omega}$
        \item $F(\omega)=\int_{\omega}u_{\omega}$ and $G(\omega)=\int_{\omega}u^{2}_{\omega}$
        \item $F(\omega)=|\omega|$ and $G(\omega)=\int_{\omega}u^{2}_{\omega}$
        \item $F(\omega)=\int_{\omega}u_{\omega}v_{\omega}$ and $G(\omega)=\int_{\omega}u^{2}_{\omega}$
        \item $F(\omega)=\int_{\omega}u_{\omega}v_{\omega}$ and $G(\omega)=\int_{\omega}u_{\omega}$
      \end{enumerate}
$u_{\omega}$ and $v_{\omega}$ being respectively the solution of $P(\omega,1)$ and $P(\omega,u_{\omega})$.
\end{proposition}

\begin{proof}
If $\Omega$ is of class $C^{2}$, then for every vector field $V\in C^{2}(\mathbb{R}^{N},\mathbb{R}^{N})$
$$dJ(\Omega,V)=\frac{dF(\Omega,V)G(\Omega)-F(\Omega)dG(\Omega,V)}{[G(\Omega)]^{2}}.$$
Then $dJ(\Omega,V)=0$ implies
\begin{enumerate}
  \item $|\nabla u_{\Omega}|=\frac{1}{J(\Omega)}\text{ on }\partial\Omega$
  \item $|\nabla v_{\Omega}|=\frac{1}{2J(\Omega)}|\nabla u_{\Omega}|\text{ on }\partial\Omega$
  \item $|\nabla v_{\Omega}||\nabla u_{\Omega}|=\frac{1}{2J(\Omega)}|\nabla u_{\Omega}|\text{ on }\partial\Omega$
  \item $|\nabla v_{\Omega}|=2J(\Omega)|\nabla u_{\Omega}|\text{ on }\partial\Omega$
  \item $|\nabla v_{\Omega}|=J(\Omega)|\nabla u_{\Omega}|\text{ on }\partial\Omega$
\end{enumerate}
Each of the overdetermined conditions obtained above says that $\Omega$ is an $N$-ball.
\end{proof}
\end{remark}
\begin{remark}
One can have the same existence result for the functional used in Section 4. if we replace $\mathcal{O}_{C}$ by the following class of open sets: For
$\varepsilon>0$
$$\mathcal{O}_{\varepsilon}=\{\omega\subset D,\;\forall x\in\partial\omega,\;\partial\omega\cap B(x,\varepsilon)\text{ has }\bar{B}'_{\varepsilon}-PGN\}$$
${B}'_{\varepsilon}$ is the $(N-1)$-ball with center $x$ \cite{ba3}.
\end{remark}
\begin{remark}
For the problem $\mathcal{QS}(f,g)$ ($g$ is positive ant it attains its maximum on $\partial C$), one can have
$$\int_{C}fdx>\int_{\partial C}gd\sigma$$
as a sufficient condition of existence.\\
Now, let $a>0$ and put $C=[-1,1]\times\{0\}\subset\mathbb{R}^2$. Consider the problem
\begin{equation*}
\mathcal{QS}(a,g) \left\{
\begin{array}{c}
-\Delta u_{\Omega}=a\delta_C \quad \text{in  }\Omega\\
u_{\Omega}=0\text{ on  }\partial \Omega\\
 -\frac{\partial u_{\Omega}}{\partial \nu }=g\text{  on  }\partial
 \Omega
\end{array}
\right.
\end{equation*}
The problem $Q_S(a,g)$ has a solution if
$a>\frac{1}{\pi}\int_{\partial B}g(\sigma)d\sigma$.
$B$ being the unit ball in $\mathbb{R}^2$.\\
Notice that in the special case where $g\equiv k=const.$, the
the condition above becomes $a>2k$ and it is necessary and sufficient condition
of existence for $\mathcal{QS}(a,k)$ \cite{ba2}.
\end{remark}

\begin{remark}
Let $\Omega$ be a bounded and convex set which contains the interior of $C$ the convex hull of the support of $f$. The existence of a domain $C_{\Omega}$ which minimizes the ratio $$R(\omega)=\frac{\int_{\omega}f}{|\partial\omega|}$$ is obtained in the class of convex subset of $\Omega$. One calls the minimum above the $f$-Cheeger set of $\Omega$ and $\lambda(\Omega)$ the $f$-Cheeger constant. Using the same arguments as in Section 3., one can prove that
\begin{enumerate}
  \item either there exists $\Omega^{*}\supsetneq\bar\Omega$ which is a solution to $\mathcal{QS}(f,\lambda(\Omega))$,
  \item or $C_{\Omega}=\Omega$, i.e $\Omega$ is $f$-Cheeger in itself.
\end{enumerate}
\end{remark}

\begin{remark}
Let $\Omega$ be an open subset of $D$. Consider the Cheeger constant $$h(\Omega)=\min_{\omega\subset \Omega}\frac{|\partial\omega|}{|\omega|}.$$
Using the domain derivative of $h$ \cite{PS}, we prove the following
\begin{proposition}
If $\int_{C}f(x)dx>\sqrt{h(C)}|\partial C|$, then the problem ($Q_S(f,\sqrt{g(\Omega)})$) admits a solution where
$$g(\Omega)=\frac{|\Omega|}{|\partial C_{\Omega}|}((N-1)H_{\partial C}-h(\Omega))\chi_{\partial\Omega\cap\partial C_{\Omega}}+h(\Omega).$$
$C_{\Omega}$ being the unique Cheeger set of $\Omega$.
\end{proposition}

\end{remark}


Mohammed Barkatou\\
Laboratoire ISTM, Département de Mathématiques\\
Université Chouaib Doukkali, Maroc\\
barkatou.m@ucd.ac.ma

\end{document}